\documentclass[12pt]{article}
 
 \usepackage[utf8]{inputenc}
 \usepackage{amsthm}
 \usepackage{enumitem}
 \usepackage{hyperref}
 \usepackage{amssymb}
 \usepackage{amsmath}
 \usepackage{dsfont}
 \usepackage{mathtools}
 \usepackage{ mathrsfs }
 \usepackage{comment}
 \usepackage[left=3.5cm,right=3.5cm,top=3.5cm,bottom=3.5cm]{geometry}

 \newtheorem{thm}{Theorem}
 
 \newtheorem{lem}{Lemma}
 \newtheorem{cor}{Corollary}
 
 \newtheorem{rem}{Remark}

\title{Profile of a self-similar growth-fragmentation}
\author{François G. Ged}
\date{}
\newcommand{\intervalleff}[2]{\left[#1\,,#2\right]}
\newcommand{\intervallefo}[2]{\left[#1\,,#2\right)}
\newcommand{\intervalleof}[2]{\left(#1\,,#2\right]}
\newcommand{\intervalleoo}[2]{\left(#1\,,#2\right)}

\begin{document}
	\maketitle
	\vspace{-1.1cm}
	\begin{center}
		\textit{Institut für Mathematik, Universität Zürich, Switzerland.}
	\end{center}
	\begin{abstract}
			A self-similar growth-fragmentation describes the evolution of particles that grow and split as time passes.
			Its genealogy yields a self-similar continuum tree endowed with an intrinsic measure.
			Extending results of Haas \cite{H04} for pure fragmentations, we relate the existence of an absolutely continuous profile to a simple condition in terms of the index of self-similarity and the so-called cumulant of the growth-fragmentation. 
			When absolutely continuous, we approximate the profile by a function of the small fragments, and compute the Hausdorff dimension in the singular case.
			Applications to Boltzmann random planar maps are emphasized, exploiting recently established connections between growth-fragmentations and random maps \cite{BBCK16,BCK15}. 
	\end{abstract}
	
	
	\textit{Classification MSC:} 60J25; 60G18; 60G30 \normalsize
	\section{Introduction}

	The main purpose of this work is to answer, for a specific family of continuous random trees (CRT in short), the following general question about measured metric spaces. If $m(r)$ denotes the measure assigned to the ball centered at some fixed distinguished point and with radius $r\geq 0$, is the non-decreasing function $m$ absolutely continuous with respect to the Lebesgue measure on $\intervallefo{0}{\infty}$ ?
	When the answer is positive, the density $m'(r)$ can then be viewed as the measure of the sphere with radius $r$. When further the metric space is a continuum tree, the density $m'$ is sometimes known as the \textit{profile} of the tree.
	
	This question has been answered by Haas \cite{H04} for the class of self-similar fragmentation trees, which notably includes Aldous' CRT. Recall that a conservative self-similar fragmentation describes the evolution of a branching particle system such that at every branching event, the sum of the masses of the children coincides with the mass of the parent, and self-similarity refers to the property that the evolution of a particle with mass $x>0$ is a scaling transformation (depending on an index $\alpha\in\mathbb{R}$) of that of a particle with unit mass. Informally, Haas and Miermont \cite{HM04} associated to a conservative self-similar fragmentation with index $\alpha<0$ a self-similar continuous random tree which is further naturally equipped with a root and a probability mass measure, and Haas \cite{H04} proved that under some very minor hypotheses, the non-decreasing function $m$ is then absolutely continuous if $\alpha>-1$, and singular if $\alpha\leq-1$.
	
	The present work should be viewed as a generalization of \cite{H04} to self-similar \textit{growth-fragmentations}, introduced by Bertoin \cite{B17}. As the name suggests, the latter extend pure fragmentations by incorporating a growth element in the dynamic of particles, and this changes deeply the behaviour of the system. Rembart and Winkel \cite{RW16} constructed recently the CRT's which describe the genealogy of self-similar growth-fragmentations with index $\alpha<0$, whereas the so-called intrinsic area measure was introduced in \cite{BBCK16}.
	
	The motivation of the present work is not just getting a formal extension of the results of Haas; it also stems from the connection between random surfaces and growth-fragmentations as we shall now explain informally. It was pointed out in \cite{BBCK16} and \cite{BCK15} that for certain random surfaces with a boundary, the process obtained by slicing the surface at fixed distances from the boundary and measuring the lengths of the resulting cycles yields a self-similar growth-fragmentation with negative index. One might then expect that, just as for smooth surfaces, the area $A(r)$ of the components at distance at most $r$ from the boundary can then be recovered by integrating the total cycle lengths at height $0\leq r'\leq r$; that is, that the non-decreasing function $r\mapsto A(r)$ is absolutely continuous with density given by the total cycle lengths. It turns out that this intuition is wrong in general, and it is thus natural to wonder whether nonetheless the absolute continuity of $A(\cdot)$ holds.
	
	The law of a growth-fragmentation is determined by the index of self-similarity and the so-called cumulant function $\kappa$. (more details are given in Section \ref{Section preliminaries}).
	Our main result is stated in terms of $\alpha$ and the smallest root $\omega_-$ of $\kappa$. More precisely, whilst the critical value is -1 for pure fragmentations, we show that the genealogical CRT of a growth-fragmentation has an absolutely continuous profile as soon as $\alpha>-\omega_-$, whereas it is singular if $\alpha\leq -\omega_-$. In particular, we shall see that for the whole family of random maps considered in \cite{BBCK16}, the function $t\mapsto A(t)$ is absolutely continuous.
	
	The paper is organized as follows. In Section \ref{Section preliminaries}, we recall the settings of \cite{BBCK16}. This includes the definition of a growth-fragmentation and its CRT. The construction of the intrinsic area measure from the branching random walk following in generations the collection of particles at birth is recalled. A loose description of the spinal decomposition is also given.
	
	Section \ref{Section Regularity of the intrinsic area process} is divided in four subsections. The first one contains our main result. The second subsection is a toolbox that recalls basic properties of the major ingredients of the proof, which is given in the third subsection. A simple corollary on the number of fragments is stated in the fourth subsection.
	
	We dwell on the absolutely continuous case in Section \ref{Section Approximation of the density} and we see that, modulo few adjustments, the proof of Haas adapts to show that the profile can be approximated by small (or equivalently relatively large) fragments.
	
	Finally, Section \ref{Section Hausdorff dimension} is devoted to the Hausdorff dimension of $\mathrm{d}A$ when singular. We obtain the lower bound from Frostman's Lemma, and derive the upper bound from the Hausdorff dimension of the leaves of the CRT, obtained by Rembart and Winkel \cite{RW16}.
	
	In Appendix are shown two technical lemmas, including the Feller property of the growth-fragmentation (Lemma \ref{Lemma Feller}), which is needed for the arguments of Haas to apply in Section \ref{Section Approximation of the density}.
	
	\section{Preliminaries}\label{Section preliminaries}
	
	\paragraph{The cell-system.}
	
	We consider a positive self-similar Markov process $X$ with index $\alpha< 0$, in the sense that its law $\mathbb{P}_x$ started from $X_0=x>0$ is the same as that of $(xX_{tx^{\alpha}})_{t\geq 0}$ under $\mathbb{P}_1$. We assume that $X$ converges almost surely to 0. Lamperti's transformation \cite{L72} enables us to view $X$ as a time changed of $\exp(\xi)$, where $\xi$ is a L\'evy process. As a consequence the lifetime of $X$, i.e. the first hitting time of the absorbing state 0, is given by an exponential functional of $\xi$ (we shall provide more details later on).

	We follow Bertoin's construction \cite{B17} of the cell-system driven by $X$: let $\chi_{\emptyset}:=(\chi_{\emptyset}(t))_{t\geq 0}$ have law $\mathbb{P}_1$. The process $\chi_{\emptyset}$ is viewed as the size of the Eve cell $\emptyset$, evolving in time. Its birth-time $b_{\emptyset}$ is taken to be 0. Let $(b_i)_{i\geq 1}$ be an exhausting sequence of its negative jump times and let $(\Delta_i)_{i\geq 1}$ be the corresponding sequence of the absolute values of the sizes of its negative jumps (the existence is ensured by the fact that $\chi_{\emptyset}$ converges to 0 almost surely).
	Each negative jump is interpreted as the birth of a new cell, that is at time $b_i$ a cell labeled $i$ is born and evolves independently of the other cells, with law $\mathbb{P}_{\Delta_i}$. The other generations are defined recursively in the same manner, using the Ulam-Harris-Neveu notation, that is every cell is labeled by some $u\in\mathbb{U}:=\bigcup_{n\geq 0}\mathbb{N}^n$. We denote $|u|$ its generation and $u(k)$ its ancestor at generation $k\leq|u|$ (by convention, $|\emptyset|=0$).
	For $x>0$, we denote $\mathcal{P}_x$ the law of the cell-system starting from a single cell of initial size $x$. Similarly, if $\underline{x}:=(x_1,x_2,\cdots)$ is a non-decreasing null sequence, $\mathcal{P}_{\underline{x}}$ is the distribution of a cell-system starting from independent cells of sizes $x_1,x_2,\cdots$.

	\paragraph{The branching random walk.}
	
	Define the collections of logarithms of cells at birth, indexed by generations, as
	\begin{align*}
		\mathcal{Z}_n:=\left\{\!\!\left\{\mathrm{ln}\chi_u(0):u\in\mathbb{N}^n\right\}\!\!\right\},
	\end{align*}
	where $\left\{\!\!\left\{\cdots\right\}\!\!\right\}$ refers to multiset, meaning that the elements are repeated according to their multiplicities.
	Thanks to self-similarity, $(\mathcal{Z}_n)_{n\geq 1}$ is a branching random walk.
	Let $(b,\sigma^2,\Lambda)$ be the characteristics of the L\'evy process $\xi$ and assume that there exists $p>0$ such that $\int_1^\infty e^{py}\Lambda(\mathrm{d}y)<\infty$ (we also assume that $\Lambda(\intervalleoo{-\infty}{0})>0$ as there are no children otherwise). We thus have that the Laplace exponent of $\xi$, given by $\psi(q):=\log\mathbb{E}(\exp(q\xi(1)))$ is finite at least on $\intervalleff{0}{p}$ (see e.g. \cite{K14} Theorem 3.6); we set $\psi(q)=\infty$ whenever the expectation is infinite. The so-called \textit{cumulant function} is defined as
	\begin{align}\label{equation cumulant kappa}
		\kappa:q\mapsto\psi(q)+\int_{\intervalleoo{-\infty}{0}}(1-e^y)^q\Lambda(\mathrm{d}y),\qquad q>0.
	\end{align}
	The mean Laplace transform of $\mathcal{Z}_1$ is then given by $q\mapsto 1-\kappa(q)/\psi(q)$, when this makes sense (see \cite{B17} Lemma 3). Hence, as soon as $\kappa(q)=0$, the process $(\sum_{|u|=n}\chi_u(0)^q)_{n\geq 0}$ is a martingale.
	We thus naturally assume that there exists $\omega_->0$ such that\footnote{In the context of branching random walks, this assumption is known as the Cram\'er hypothesis and $\omega_-$ is called the Malthusian parameter.}
	\begin{align}\label{Cramer hypothesis}
		\kappa(\omega_-)=0,\qquad
		-\infty<\kappa'(\omega_-)<0.
	\end{align}
	The so-called \textit{intrinsic martingale} introduced in \cite{BBCK16} is then defined as
	\begin{align*}
		\mathcal{M}(n):=\sum_{|u|=n}\chi_u(0)^{\omega_-},\qquad n\geq 0.
	\end{align*}
	This martingale is moreover uniformly integrable with mean 1 under $\mathcal{P}_1$ (see \cite{BBCK16} Lemma 2.3).
	We shall also denote $\omega_+:=\sup\left\{q\geq 0:\kappa(q)<0\right\}$, which is strictly greater than $\omega_-$ thanks to \eqref{Cramer hypothesis}. (In \cite{BBCK16}, $\omega_+$ is a second root of $\kappa$, which if it exists, is consistent with our definition.)
	Finally, we rule out the case where $X$ is the negative of a subordinator, as this induces fragmentation processes, which are fully addressed by \cite{H04}.

	\paragraph{The Ulam tree, the CRT and the intrinsic area measure.}
	
	In \cite{BBCK16}, the authors define a random measure on the boudary of the Ulam tree $\partial\mathbb{U}$, which is the set of infinite integer sequences, endowed with the distance $d(\ell,\ell'):=\exp(-\sup\{n\geq 0:\ell(n)=\ell'(n)\})$ which makes it a complete metric space (recall that $\ell(n)$ denotes the ancestor of $\ell$ at generation $n$). Specifically, for every $u\in\mathbb{U}$ with $|u|=n$, let $B(u):=\{\ell\in\partial\mathbb{U}:\ell(n)=u\}$ be the ball in $\partial\mathbb{U}$ generated by $u$. The \textit{intrinsic area measure} on $\partial\mathbb{U}$ is then defined by
	\begin{align*}
		\mathcal{A}(B(u)):=\lim_{k\to\infty}\sum_{|v|=k,v(n)=u}\chi_v(0)^{\omega_-}.
	\end{align*}
	(This is well-defined thanks to the uniform integrability of $(\mathcal{M}(n))_{n\geq 0}$ and the branching property.)
	The total mass is denoted $\mathcal{M}:=\lim_{n\to\infty}\mathcal{M}(n)=\mathcal{A}(\partial\mathbb{U})$.
	
	Rembart and Winkel \cite{RW16} built a CRT from the cell-system, that is very similar to $\partial\mathbb{U}$.
	The construction is as follows: construct a first segment of length equal to the lifetime $\zeta_{\emptyset}$ of $\emptyset$, endowed with a metric corresponding to the age of the cell. It means that each point of this branch corresponds to the Eve cell at a particular time of its life; the root $\rho$ is thus naturally taken to be the point corresponding to 0. On this branch, at every jump location $b_i$, glue a new branch of length equal to the lifetime $\zeta_i$ of the cell $i$, with the corresponding metric. This yields a CRT $(\mathcal{T}_1,d_1)$. For all $n\geq 1$, to obtain $(\mathcal{T}_{n+1},d_{n+1})$, repeat this procedure on every branch $u\in\mathbb{N}^n$ at locations $\left\{b_{uj}:j\geq 1\right\}$.
	
	Theorem 1.7 in \cite{RW16} shows that, whenever $\psi(-\alpha)<0$, $(\mathcal{T}_n,d_n)_{n\geq 1}$ converges almost surely in the Gromov-Hausdorff topology to some compact CRT $(\mathcal{T},d)$. Eventhough it is not explicitely given in their construction, there is a very natural way to define simultaneously the analogue of the intrinsic area measure on $\mathcal{L}(\mathcal{T})$, the set of leaves of $\mathcal{T}$, that we now introduce. Fix $n\geq 1$ and consider $\mathcal{T}_n$. For every $u\in\mathbb{N}^n$ and $j\geq 1$, put a mass $\chi_{uj}(0)^{\omega_-}$ at location $b_{uj}$ on the branch of $u$. This defines a measure $\mathcal{A}_n$ on $\mathcal{T}_n$, with total mass given by $\mathcal{M}(n+1)$. As for the Ulam tree, it is clear that $(\mathcal{A}_n)_{n\geq 1}$ converges weakly toward a measure $\mathcal{A}_{\mathcal{T}}$ with total mass $\mathcal{M}$ and supported on $\mathcal{L}(\mathcal{T})$, the set of leaves of $\mathcal{T}$. The correspondance between $\mathcal{T}$ and $\overline{\mathbb{U}}:=\mathbb{U}\cup\partial\mathbb{U}$ is straightforward from the two constructions, that is every $x\in\mathcal{T}$ corresponds to either a unique $\chi_u(t)$ for some $u\in\mathbb{U}, t\in\intervalleff{0}{\zeta_u}$, or a unique $\ell\in\partial\mathbb{U}$. In particular, $\mathcal{A}_{\mathcal{T}}$ and $\mathcal{A}$ are essentially the same, this is even clearer when looking at the masses at heights. 
	
	Recall that the height function on $\mathcal{T}$ is defined as the distance to the root
	\begin{align*}
		\mathrm{ht}(x):=d(\rho,x).
	\end{align*}
	We then define $A_{\mathcal{T}}:\mathbb{R}_+\to\mathbb{R}_+$ by 
	\begin{align}\label{equation def A}
	&A_{\mathcal{T}}:t\mapsto\mathcal{A}_{\mathcal{T}}(\left\{\ell\in\mathcal{L}(\mathcal{T}):\mathrm{ht}(\ell)\leq t\right\}).\nonumber\\
	\shortintertext{This coincides exactly with}
	&A:t\mapsto\mathcal{A}(\left\{\ell\in\partial\mathbb{U}:\zeta_\ell\leq t\right\}),
	\end{align}
	where $\zeta_\ell:=\lim_{n\to\infty}b_{\ell(n)}$. (Actually, $\mathcal{L}(\mathcal{T})$ also contains cells at death-times, but they do not generate area since there are only countably many.)
	Since the cell system carries more informations, we shall rather work with $A$ than $A_{\mathcal{T}}$.
	
	The elements of $\mathcal{T}\setminus\mathcal{L}(\mathcal{T})$ at a fixed height $t\geq 0$ correspond to the collection of cells alive at time $t$:
	\begin{align*}
		\mathbf{X}(t):=\left\{\!\!\left\{\chi_u(t-b_u):u\in\mathbb{U},b_u\leq t<b_u+\zeta_u\right\}\!\!\right\}.
	\end{align*}
	This is the definition of Bertoin \cite{B17} of the growth-fragmentation process induced by the cell-system.
	Shi \cite{S17} showed that the distribution of $\mathbf{X}$ is characterized by the pair $(\kappa,\alpha)$.\footnote{However, this is not the case of the distribution of the cell-system.}
	The lifetime of $\mathbf{X}$ is defined as $\zeta:=\inf\left\{t>0:\mathbf{X}(t)=\emptyset\right\}$.
	In \cite{B17}, it is shown that the Cram\'er hypothesis \eqref{Cramer hypothesis} ensures that the following properties hold:
	\begin{enumerate}[label=\textbullet]
		\item Almost surely, for any fixed $\epsilon>0$, there are finitely many fragments larger than $\epsilon$ in $\mathbf{X}(t)$ for all $t\geq 0$.
		\item $\zeta<\infty$ almost surely. (\cite{B17} Corollary 3)
		\item $\mathbf{X}$ enjoys the self-similarity and branching properties, as stated in \cite{B17} Theorem 2.
	\end{enumerate}
	As a consequence, $\mathcal{T}$ satisfies a Markov-branching type property, that we express in terms of $A$ as follows: let $t\geq 0$ and let $(A_i)_{i\geq 1}$ be a sequence of i.i.d. copies of $A$, independent of $(\mathbf{X}(u);u\leq t)$, then for all $s\geq 0$:
	\begin{align}\label{equation branching + self-similarity of A}
		A(t+s)-A(t)\stackrel{d}{=}\sum_{i\geq 1}X_i(t)^{\omega_-}A_i(sX_i(t)^{\alpha}),
	\end{align}
	where for $i\geq 1$, $X_i(t)$ denotes the size of the $i$th largest fragment in $\mathbf{X}(t)$ (being possibly 0).

	\paragraph{Spinal decomposition}
	We now give an informal description of the spinal decomposition induced by $\mathcal{A}$, introduced in \cite{BBCK16} Section 4. The statements are provided without proof, the reader is refered to this paper for a rigorous treatment.
	
	We introduce a probability measure $\widehat{\mathcal{P}}_1$ describing the joint distribution of a cell-system and a random leaf $\sigma\in\partial\mathbb{U}$. Under $\widehat{\mathcal{P}}_1$, the law of the cell-system is absolutely continuous with respect to $\mathcal{P}_1$, with density $\mathcal{M}$. The random leaf $\sigma$ is then tagged according to the intrinsic area. In particular we have
	
	\begin{lem}\label{Lemma A and random leaf}
		Under $\widehat{\mathcal{P}}_1$ and conditionally on the cell-system, the probability measure $\mathrm{d}A/\mathcal{M}$ satisfies
		\begin{align*}
			\frac{\mathrm{d}A(t)}{\mathcal{M}}
			=\widehat{\mathcal{P}}_1\left(\zeta_\sigma\in\mathrm{d}t|(\chi_u)_{u\in\mathbb{U}}\right).
		\end{align*}
	\end{lem}
	
	Let $\phi:q\mapsto\kappa(\omega_-+q),\ q\geq 0$. It is known that $\phi$ can be viewed as the Laplace exponent of a L\'evy process\footnote{This fact is stated in \cite{BBCK16} Lemma 2.1 for $q\mapsto\kappa(\omega_++q)$, however it is also true for $\phi$ by the same arguments.} that we denote $\eta$. We then define the positive self-similar Markov process $(Y_t)_{t\geq 0}$ with index $\alpha$, associated with $\eta$ by Lamperti's transformation, that is
	\begin{align*}
		(Y_t)_{t\geq 0}:=\left(\exp\left(\eta(\tau_t)\right)\right)_{t\geq 0},
	\end{align*}
	where the time-change $\tau_t$ is defined for all $t\geq 0$ by
	\begin{align}\label{Definition time-change tau Lamperti}
		\tau_t:=\inf\left\{s\geq 0:\int_0^s e^{-\alpha\eta(u)}\mathrm{d}u\geq t\right\}.
	\end{align}
	The absorption time of $Y$ is thus given by the following exponential functional
	\begin{align}\label{equation I exponential functional}
		I=\int_0^\infty e^{-\alpha\eta(t)}\mathrm{d}t.
	\end{align}
	Since $\kappa'(\omega_-)<0$ by \eqref{Cramer hypothesis}, we know that $\eta$ drifts to $-\infty$ and $I<\infty$ almost surely. We shall denote $\widehat{\mathbb{P}}_x$ the law of $Y$ starting from $x>0$.
	
	The \textit{spine} $(\sigma(t))_{t\geq 0}$ is the process following the size of the ancestors of $\sigma$ in time. Remark that we can write $\zeta_{\sigma}=\inf\left\{t>0:\sigma(t)=0\right\}$. We thus call $\zeta_\sigma$ the \textit{lifetime} of $\sigma$ to emphasize that we will look at $\sigma$ as a random process rather than a random element of a random metric space. In this direction, we have
	
	\begin{lem}\label{Lemma spine distributed as Y}
		Under $\widehat{\mathcal{P}}_1$, the spine $\sigma$ is distributed as $Y$ under $\widehat{\mathbb{P}}_1$. In particular, it holds that $\zeta_{\sigma}\stackrel{d}{=}I$.
	\end{lem}
	
	Lemma \ref{Lemma A and random leaf} relates $A$ to the lifetime of the spine, which in turns is distributed as the variable $I$ by Lemma \ref{Lemma spine distributed as Y}. Let $\mathcal{C}_0^\infty(\mathbb{R}_+^*)$ be the set of infinitely differentiable functions on $\mathbb{R}_+^*$ vanishing together with their derivatives at infinity.  
	Equation \eqref{equation I exponential functional} plays a crucial role to obtain distributional properties of $I$. The next lemma collects some that we shall extensively use throughout the rest of this work.
	
	\begin{lem}\label{properties of k}
		The variable $I$ has a bounded density in $\mathcal{C}_0^\infty(\mathbb{R}_+^*)$, which we denote by $k$. Further, $\lim_{x\to 0+}k(x)=0$ and
	\begin{align*}
			\widehat{\mathbb{E}}_1\left(I^{-1}\right)=\alpha\kappa'(\omega_-)<\infty.
		\end{align*}
	\end{lem}
	These results are already known. Hence in the following proof, we only provide references and check that the hypotheses of the cited theorems are fulfilled.
	\begin{proof}
	Theorem 3.9 in \cite{BLM08} ensures the existence of $k$.
	Recently, Patie and Savov \cite{PS16} have shown that $k$ is infinitely differentiable. 
	The L\'evy measure $\Pi$ of $\eta$ is given (see \cite{BBCK16} Section 4.3) by 
	\begin{align}\label{equation measure Pi}
		\Pi(\mathrm{d}y):=e^{\omega_-y}(\Lambda+
		\widetilde{\Lambda})(\mathrm{d}y).
	\end{align}
	where $\widetilde{\Lambda}$ is the push-forward of $\Lambda$ by $y\mapsto\mathrm{1}_{\left\{y<0\right\}}\log(1-e^y)$.
	We see that if $\Lambda(\mathbb{R}_-)=\infty$, then $\Pi$ also has infinite total mass. Notice that we have either $\Lambda(\mathbb{R}_-)=\infty$, or $\sigma^2>0$, or that $\eta$ is a compound Poisson process with a non-negative drift. Theorem 2.4.(3)\footnote{In \cite{PS16}, the authors use the equivalent convention that the process drifts to $+\infty$ and they take the negative of the exponential to define $I$.} in \cite{PS16} thus shows that $k\in\mathcal{C}_0^\infty(\mathbb{R}_+^*)$ (see in particular Remark 2.5 in the same paper). Finally, the limit at 0 of $k$ is given in \cite{PS16} by Theorem 2.15. 
	The statement on the moment of order $-1$ can be found in \cite{CPY97} Proposition 3.1(iv).
	\end{proof}
	
	We conclude this section by recalling the following essential fact on the spinal decomposition: conditionally on $(\sigma(t))_{t\geq 0}$, a child depends only on the spine through its own initial value, given by the size $x$ of the negative jump who generated it, and then evolves with law $\mathbb{P}_x$, independently from $(\sigma(t))_{t\geq 0}$ and the other children.

	\underline{\textit{Notation}}:
	In the sequel, the expectations under $\mathbb{P}_x,\widehat{\mathbb{P}}_x,\mathcal{P}_x,\widehat{\mathcal{P}}_x$ are denoted respectively by $\mathbb{E}_x,\widehat{\mathbb{E}}_x,\mathcal{E}_x,\widehat{\mathcal{E}}_x$.

	\section{Existence of the profile}\label{Section Regularity of the intrinsic area process}
	
	\subsection{Main result}
	
	The following theorem answers the question of the regularity of $t\mapsto A(t)$ in terms of $\alpha$.
	
	\begin{thm}\label{absolute continuity}
	$\mathrm{d}A$ is almost surely singular with respect to the Lebesgue measure if and only if $\alpha\leq -\omega_-$, whereas $\mathrm{d}A(x)$ is absolutely continuous almost surely whenever $\alpha>-\omega_-$.
	\end{thm}
	
	We recover Theorem 4 of Haas \cite{H04}, since in the pure fragmentations case $\omega_-=1$ (her result does not require dislocations to be binary though). Recall that this theorem can be read as a statement on the $A$ stemming from either $\partial\mathbb{U}$ or $\mathcal{T}$, as explained in Section \ref{Section preliminaries}.
	
	\subsection{Toolbox}\label{Subsection toolbox}
	We introduce in this subsection the main tools of the proof of Theorem \ref{absolute continuity}.
	Let $\mu$ be a measure on $\mathbb{R}$. We denote its Fourier-Stieljes transform
	\begin{align*}
		\mathcal{F}_\mu(\theta):=\int_{\mathbb{R}}e^{i\theta x}\mu(\mathrm{d}x),\qquad \theta\in\mathbb{R}.
	\end{align*}
	Recall from Plancherel's Theorem that
	\begin{equation}\label{condition density}		
	\mu(\mathrm{d}x)\ll\mathrm{d}x\quad\text{with}\quad\mu(\mathrm{d}x)/\mathrm{d}x\in\mathbb{L}^2(\mathrm{d}x)\qquad\Leftrightarrow\qquad\mathcal{F}_{\mu}\in\mathbb{L}^2(\mathrm{d}x).
	\end{equation}
	We shall use \eqref{condition density} to prove the next lemma, which is the main ingredient in the proof of the absolute continuity of $\mathrm{d}A$.
	It was used in \cite{H04} but somewhat implicitely. We state it in a general setting. 
	
	Let $\mathbf{P}$ be a probability measure on a generic random space. Let $(E,d,\mu,\rho)$ be a random measured metric space, where $\mu$ is a measure on $E$ with finite total mass $\mathbf{P}$-almost surely and $\rho\in E$ is a distinguished element.
	Let $B(\rho,r)$ be the open ball centered in $\rho$ with radius $r>0$.
	Let $\gamma,\gamma'$ be two random variables in $E$ such that $\gamma$ and $\gamma'$ are conditionally independent given $(E,d,\mu,\rho)$, with conditional law $\mu(\cdot)/\mu(E)$.
	
	\begin{lem}\label{Lemma absolute continuity for random metric space}
		If the law of $\nabla:=d(\rho,\gamma)-d(\rho,\gamma')$ has a density $h$ which is bounded in a neighbourhood of 0, then $m:r\mapsto \mu(B(\rho,r))$ is absolutely continuous with respect to the Lebesgue measure, with a density in $\mathbb{L}^2(\mathrm{d}x)$ $\mathbf{P}$-almost surely.  
	\end{lem}

	\begin{proof}
		Recall \eqref{condition density}, we thus look at
		\begin{align*}
			\frac{|\mathcal{F}_{\mathrm{d}m}(\theta)|^2}{\mu(E)^2}
			&=\frac{\mathcal{F}_{\mathrm{d}m}(\theta)}{\mu(E)}\cdot\frac{\mathcal{F}_{\mathrm{d}m}(-\theta)}{\mu(E)}
			=\int_0^\infty\int_0^\infty\frac{\mathrm{d}m(x)}{\mu(E)}\cdot\frac{\mathrm{d}m(y)}{\mu(E)}e^{i\theta(x-y)}\\
			&=\mathbf{E}\left(e^{i\theta\nabla}|(E,d,\mu,\rho)\right),
		\end{align*}
		where $\mathbf{E}$ is the expectation operator induced by $\mathbf{P}$.
		We see in particular that $\theta\mapsto\mathbf{E}\left(e^{i\theta\nabla}\right)\geq 0$. Theorem 9 in \cite{BC49} ensures that if $\nabla$ has a density bounded in a neighbourhood of $0$, then its Fourier transform is integrable, that is
		\begin{align*}
			\int_{\mathbb{R}}\mathbf{E}\left(e^{i\theta\nabla}\right)\mathrm{d}\theta=\mathbf{E}\left(\int_{\mathbb{R}}\frac{|\mathcal{F}_{\mathrm{d}m}(\theta)|^2}{\mu(E)^2}\right)<\infty.
		\end{align*}
		We conclude by Plancherel's Theorem \eqref{condition density}.
	\end{proof}
	
	We shall use this Lemma for two suitable choices of $\mathbf{P}$, taking $(E,d,\mu,\rho)$ as $(\mathcal{T},d,\mathcal{A},0)$, the distance $d$ being the age, see Section \ref{Section preliminaries}. This means in particular that $m=A$.

	We state an easy but important consequence of Lemma \ref{Lemma A and random leaf} in the next lemma. Recall that $\zeta$ denotes the lifetime of $\mathbf{X}$.
	
	\begin{lem}\label{Lemma A(epsilon)}
		The function $t\mapsto A(t)$ is strictly increasing on $\intervalleoo{0}{\zeta}$ and it holds that
		\begin{align*}
			\mathcal{E}_1\left(A(\epsilon)\right)=o(\epsilon),\quad\epsilon\to 0.
		\end{align*}
	\end{lem}
	
	\begin{proof}
		Fix $\epsilon>0$. We write
		\begin{align*}
			\mathcal{E}_1(A(\epsilon))
			&=\widehat{\mathcal{E}}_1\left(\frac{A(\epsilon)}{\mathcal{M}}\right)=\widehat{\mathcal{P}}_1\left(\zeta_{\widehat{\chi}}\leq \epsilon\right),
		\end{align*}
		where the last identity is seen from Lemma \ref{Lemma A and random leaf}.
		Lemma \ref{Lemma spine distributed as Y} combined with the fact that $k(x)\to 0$ as $x\to 0+$ from Lemma \ref{properties of k} entail that
		\begin{align*}
			\mathcal{E}_1\left(A(\epsilon)\right)
			=o(\epsilon),\quad\epsilon\to 0.
		\end{align*}
		Very similar arguments as in \cite{H04} Proposition 10$(iv)$ can be applied to show that $A$ is strictly increasing on $\intervalleoo{0}{\zeta}$.
	\end{proof}

	\subsection{Proof of Theorem \ref{absolute continuity}}\label{Subsection proof}
	 
	In the case of self-similar pure fragmentations, Haas \cite{H04} exploited the unit interval representation to tag two fragments by sampling two independent uniform random variables on $\intervalleff{0}{1}$. In the context of Lemma \ref{Lemma absolute continuity for random metric space}, it means that the measure $\mu$ is uniform over the leaves, given the tree.
	Recall that we work on $\overline{\mathbb{U}}$. In our case, $\mathcal{A}$ is not uniform on $\partial\mathbb{U}$. However it is not required to apply Lemma \ref{Lemma absolute continuity for random metric space}.
	
	We divide the proof into two subsections. Even though the second one would be enough to prove Theorem \ref{absolute continuity} in great generality, it is very similar to the first one but involves considerations that can be avoided in some cases, and we do so for the sake of clarity.
	
	\subsubsection{The case $\omega_+/\omega-_->2$ and $\alpha>-\omega_-$.}
	
	We assume throughout this subsection that $\omega_+/\omega_->2$ and $\alpha>-\omega_-$. The reason is that thanks to Lemma 2.3 in \cite{BBCK16}, $\mathcal{E}_1(\mathcal{M}^2)<\infty$. We can thus define a probability measure $\check{\mathcal{P}}_x$ absolutely continuous with respect to $\widehat{\mathcal{P}}_x$ with density $\mathcal{M}/\widehat{\mathcal{E}}_x(\mathcal{M})=x^{\omega_-}\mathcal{M}/\mathcal{E}_x(\mathcal{M}^2)$, which also means that $\check{\mathcal{P}}_x$ has density $\mathcal{M}^2/\mathcal{E}_x(\mathcal{M}^2)$ with respect to $\mathcal{P}_x$. In particular, we can choose $\mathbf{P}=\check{\mathcal{P}}_1$ and try to apply Lemma \ref{Lemma absolute continuity for random metric space}. (This argument does not apply when $\omega_+/\omega_-\leq 2$ since we then know, still from \cite{BBCK16} Lemma 2.3, that $\mathcal{E}_1(\mathcal{M}^2)=\infty$.)
	
	In this subsection, we write $C=1/\mathcal{E}_1(\mathcal{M})$.
	
	\begin{lem}\label{Lemma density nabla integrable case}
		Consider $\nabla$ as in Lemma \ref{Lemma absolute continuity for random metric space}. Under $\check{\mathcal{P}}_1$, $\nabla$ has a density $h:\mathbb{R}\to\mathbb{R}_+\cup\left\{\infty\right\}$, given by
	\begin{align*}
		h(x)=C\widehat{\mathbb{E}}\left(\sum_{s>0}|\Delta_-Y(s)|^{\omega_-+\alpha}Y(s)^{\alpha}\int_0^\infty\mathrm{d}uk(u|\Delta_-Y(s)|^{\alpha})k((u-x)Y(s)^{\alpha})\right),
	\end{align*}
	where $k$ denotes the density of the law of $I$ from Lemma \ref{properties of k}.
	\end{lem}
	
	\begin{proof}
	In what follows, we shall implicitely use several times Tonelli's Theorem, since every terms involved in the proof are positive.
	Let $f:\mathbb{R}\to\mathbb{R}_+$ be any non-negative measurable function. We have
	\begin{align*}
		\check{\mathcal{E}}_1\left(f(\nabla)\right)
		&=C\widehat{\mathcal{E}}_1\left(\mathcal{M}f(\nabla)\right)
		=C\widehat{\mathcal{E}}_1\left(\mathcal{M}f(\zeta_\sigma-\zeta_{\sigma'})\right)\\
		&=C\widehat{\mathcal{E}}_1\left(\mathcal{M}\widehat{\mathcal{E}}_1\left(\left.f(\zeta_\sigma-\zeta_{\sigma'})\right|(\chi_u)_{u\in\mathbb{U}},\sigma\right)\right)\\
		&=C\widehat{\mathcal{E}}_1\left(\int_{\partial\mathbb{U}}\mathcal{A}(\mathrm{d}\ell)f(\zeta_\sigma-\zeta_{\ell})\right),
	\end{align*}
	where we used the conditional independence of $\sigma$ and $\sigma'$ given $(\chi_u)_{u\in\mathbb{U}}$. We now choose a subtree among those generated by the children of the spine $(\sigma(s))_{s\geq 0}$ (the ancestors of $\sigma$), according to the intrinsic area. Denoting $\partial\mathbb{U}_s$ the leaves of the tree descending from the negative jump (if any) of the spine at time $s$, it reads as
	\begin{align*}
		\check{\mathcal{E}}_1\left(f(\nabla)\right)
		&=C\widehat{\mathcal{E}}_1\left(\sum_{s>0}\widehat{\mathcal{E}}_1\left(\left.\mathcal{M}_s\int_{\partial\mathbb{U}_s}\frac{\mathcal{A}(\mathrm{d}\ell)}{\mathcal{M}_s}f(\zeta_\sigma-\zeta_{\ell})\right|(\sigma(t))_{t\geq 0}\right)\right)\\
		&=C\widehat{\mathcal{E}}_1\left(\sum_{s>0}\widehat{\mathcal{E}}_{1}\left(\left.f(\zeta_\sigma-s-\zeta_{\widehat{\sigma}})\right|(\sigma(t))_{t\geq 0}\right)\right),
	\end{align*}
	where $\widehat{\sigma}$ is a random leaf of $\partial\mathbb{U}_s$ tagged according to the restiction of the intrinsic area on this subtree. Hence, $\zeta_{\widehat{\sigma}}$ is distributed as the lifetime of a spine under $\widehat{\mathcal{P}}_{|\Delta_-\sigma(s)|^{\omega_-}}$, conditionally on $(\sigma(t))_{t\geq 0}$. More precisely, Theorem 4.7 in \cite{BBCK16} ensures that $\zeta_{\widehat{\sigma}}$ depends on $\sigma$ only through $\Delta_-\sigma(s)$, and is independent of what happens to the spine at other times. The scaling property yields that the above is equal to
	\begin{align}\label{Equation second spine rescaled}
		&C\widehat{\mathcal{E}}_1\left(\sum_{s>0}|\Delta_-\sigma(s)|^{\omega_-}\widehat{\mathcal{E}}_{1}\left(\left.f\left(\zeta_\sigma-s-|\Delta_-\sigma(s)|^{-\alpha}\zeta_{\widehat{\sigma}}\right)\right|(\sigma(t))_{t\geq 0}\right)\right),
	\end{align}
	where now the law of $\zeta_{\widehat{\sigma}}$ is independent of $\sigma$ and is that of $I$ under $\widehat{\mathbb{P}}_1$ by Lemma \ref{Lemma spine distributed as Y}. We thus get
	\begin{align*}
		\check{\mathcal{E}}_1\left(f(\nabla)\right)
		&=C\widehat{\mathcal{E}}_1\left(\sum_{s>0}|\Delta_-\sigma(s)|^{\omega_-}\int_0^\infty\mathrm{d}xk(x)f\left(\zeta_\sigma-s-|\Delta_-\sigma(s)|^{-\alpha}x\right)\right)\\
		&=C\widehat{\mathbb{E}}_1\left(\sum_{s>0}|\Delta_-Y(s)|^{\omega_-}\int_0^\infty\mathrm{d}xk(x)f\left(I-s-|\Delta_-Y(s)|^{-\alpha}x\right)\right),
	\end{align*}
	by Lemma \ref{Lemma spine distributed as Y}.
	Since every term in the sum is positive, we can use the optional projection Theorem (\cite{DM2} Theorem 57) with respect to the natural filtration of $Y$ (see also Theorem 43 in the same book for a definition of optional projection).
	The right-hand side above becomes
	\begin{align*}
		&C\widehat{\mathbb{E}}_1\left(\sum_{s>0}|\Delta_-Y(s)|^{\omega_-}\int_0^\infty\mathrm{d}xk(x)\int_0^\infty\mathrm{d}yk(y)f\left(Y(s)^{-\alpha}y-|\Delta_-Y(s)|^{-\alpha}x\right)\right)\\
		&\hspace{2cm}=C\widehat{\mathbb{E}}_1\Bigg(\sum_{s>0}|\Delta_-Y(s)|^{\omega_-+\alpha}Y(s)^{\alpha}\int_0^\infty\mathrm{d}xk(x|\Delta_-Y(s)|^{\alpha})\\
		&\hspace{7.5cm}\times\int_0^\infty\mathrm{d}yk(yY(s)^{\alpha})f\left(y-x\right)\Bigg),
	\end{align*}
	which gives the claim.
	\end{proof}

	We now provide the proof of the absolute continuity of $\mathrm{d}A$.
	
	\begin{proof}[Proof of Theorem \ref{absolute continuity}, case $\alpha>\omega_-$ and $\omega_+/\omega_->2$.]
		By the lemmas \ref{Lemma absolute continuity for random metric space} and \ref{Lemma density nabla integrable case}, it is sufficient to show that the supremum of $h$ is finite, that is
		\begin{align*}
			\sup_{x\in\mathbb{R}}\widehat{\mathbb{E}}\left(\sum_{s>0}|\Delta_-Y(s)|^{\omega_-+\alpha}Y(s)^{\alpha}\int_0^\infty\mathrm{d}uk(u|\Delta_-Y(s)|^{\alpha})k((u-x)Y(s)^{\alpha})\right)<\infty.
		\end{align*}
		For all $x\in\mathbb{R}$, we can bound $h(x)$ after a suitable change of variable by
	\begin{align}\label{Equation upper bound h abs cont}
		h(x)
		\leq C\widehat{\mathbb{E}}\left(\sum_{s>0}|\Delta_-Y(s)|^{\omega_-}(Y(s)\vee|\Delta_-Y(s)|)^{\alpha}\right)||k||_\infty \int_0^\infty k(u)\mathrm{d}u,
	\end{align}
	the last integral being equal to 1 since $k$ is a density.
	It remains to show that the expectation is finite. Since the terms in the sum do not depend on $\alpha$, we can assume without loss of generality that $Y$ is homogeneous, that is $Y(s)=\exp(\eta(s))$ for all $s\geq 0$. The compensation formula for Poisson point processes then yields that the expectation in the right-hand side above is equal to
	\begin{align*}
		&\int_0^\infty\mathrm{d}s\widehat{\mathbb{E}}_1\left(e^{(\omega_-+\alpha)\eta(s)}\right)\int_{\intervalleoo{-\infty}{0}}\Pi(\mathrm{d}y)(1-e^y)^{\omega_-}(e^y\vee(1-e^y))^{\alpha}\\
		&\hspace{3cm}\leq\int_0^\infty\mathrm{d}se^{\kappa(2\omega_-+\alpha)s}\int_{\intervalleoo{-\infty}{0}}\Lambda(\mathrm{d}y)e^{\omega_-y}(1-e^y)^{\omega_-}2^{-\alpha}.
	\end{align*}
	The last integral is finite by definition of $\omega_-$, as well as the first one when $\kappa(2\omega_-+\alpha)<0$, that is $-\omega_-<\alpha<\omega_+-2\omega_-$. Since $\alpha>-\omega_-$ and $\omega_+/\omega_->2$ by assumption, we always have $-\omega_-<\alpha<0<\omega_+-2\omega_-$, which ends the proof.	
	\end{proof}

	\subsubsection{The case $\omega_+/\omega_-\leq 2$ and $\alpha>-\omega_-$.}
	
	The main issue when $\omega_+/\omega_-\leq 2$ is that the measure $\check{\mathcal{P}}$ defined earlier has now infinite total mass, which prevents us to use Lemma \ref{Lemma absolute continuity for random metric space}. We overcome this by defining a new measure $\check{\mathcal{P}}^{(K)}$ in such a way that the two random leaves are tagged among those whose ancestors' sizes have never been too large, which, we will see, entails that $\check{\mathcal{P}}^{(K)}$ has finite total mass.
	In this direction, let $K>0$ intended to tend to $\infty$. Define $B_K$ as
	\begin{align*}
		B_K:=\left\{\ell\in\partial\mathbb{U}:\ell^*\leq K\right\},
	\end{align*}
	where $\ell^*:=\sup_{t\geq 0}\ell(t)$. Let $\mathcal{A}_K$ be the restriction of the measure $\mathcal{A}$ to $B_K$, and $\mathcal{M}_K:=\mathcal{A}_K(B_K)$.
	We now define for all $x>0$ the probability measure $\widehat{\mathcal{P}}_x^{(K)}$ such that it is absolutely continuous with respect to $\mathcal{P}_x$, with density $\mathcal{M}_K/\mathcal{E}_x(\mathcal{M}^{(K)})$. In the same vein, we would like to define $\check{\mathcal{P}}_x^{(K)}$ to be absolutely continuous with respect to $\widehat{\mathcal{P}}_x^{(K)}$ with density $\mathcal{M}_K/\widehat{\mathcal{E}}_x^{(K)}(\mathcal{M}_K)$. To ease the expressions, we shall write $C$ for the finite and strictly positive constants that appear when changing of measures. Even if $C$ may vary from line to line, it is always known and only depends on the starting point and $K$, which will play no role in the proofs. 
	We need the following
	
	\begin{lem}\label{Lemma 2nd moment truncated intrinsic area}
		For all $x>0$, it holds that $\widehat{\mathcal{E}}_x^{(K)}(\mathcal{M}_K)=\mathcal{E}_x(\mathcal{M}_K^2)<\infty$. In particular, $\left\{\check{\mathcal{P}}_x^{(K)};x>0\right\}$ is a family of probability measures.
	\end{lem}
	
	\begin{proof}
		By self-similarity, it is enough to show that $\mathcal{E}_1(\mathcal{M}_K^2)<\infty$. Let $\sigma_K$ denote a random leaf sampled on $B_K$ with conditional law $\mathcal{A}_K(\cdot)/\mathcal{M}_K$ given $(\chi_u)_{u\in\mathbb{U}}$. Similarly as in the previous subsection, we write 
		\begin{align*}
		\mathcal{E}_1\left(\mathcal{M}_K^2\right)
		&=C\widehat{\mathcal{E}}_1^{(K)}\left(\mathcal{M}_K\right)
		=C\widehat{\mathcal{E}}_1^{(K)}\left(\sum_{s>0}|\Delta_- \sigma_K(s)|^{\omega_-}\widehat{\mathcal{E}}_1^{(K)}\left(\mathcal{M}_{K,s}|(\sigma_K(t))_{t\geq 0}\right)\right)\\
		\shortintertext{where $\mathcal{M}_{K,s}$ is the rescaled truncated mass generated by the cell born at time $s$ (truncated when cells reach a size $K/|\Delta_-\sigma_K(s)|^{\omega_-}$ by the scaling property). Since $\mathcal{M}_{K,s}$ cannot be greater than the non-truncated area, that has expectation 1, we get}
		\mathcal{E}_1\left(\mathcal{M}_K^2\right)
		&\leq C\widehat{\mathcal{E}}_1^{(K)}\left(\sum_{s>0}|\Delta_- \sigma_K(s)|^{\omega_-}\right)=C\widehat{\mathcal{E}}_1\left(\sum_{s>0}\mathds{1}_{\left\{\sigma^*\leq K\right\}}|\Delta_- \sigma_K(s)|^{\omega_-}\right)\\
		&=C\widehat{\mathbb{E}}_1\left(\sum_{s>0}\mathds{1}_{\left\{Y^*\leq K\right\}}|\Delta_- Y(s)|^{\omega_-}\right).
		\end{align*}
		As previously we assume without loss of generality that $Y$ is homogeneous. Now we fix $p\in\intervalleoo{0}{\omega_+-\omega_-}$ and we bound the latter from above by
		\begin{align*}
			CK^{\omega_--p}\widehat{\mathbb{E}}_1\left(\sum_{s>0}e^{p\eta(s-)}(1-e^{\Delta_-\eta(s)})^{\omega_-}\right).
		\end{align*}
		The compensation formula yields that the expectation is equal to
		\begin{align*}
			\int_0^\infty\mathrm{d}se^{(\omega_-+p)\kappa(s)}\int_{\intervalleoo{-\infty}{0}}\Pi(\mathrm{d}y)(1-e^y)^{\omega_-},
		\end{align*}
		which is finite since $\kappa(\omega_-+p)<0$ and by definition of $\omega_-$ and $\Pi$. This shows that $\mathcal{E}_x(\mathcal{M}^2_K)<\infty$ for any $x>0$ by self-similarity.
	\end{proof}
	
	Thanks to Lemma \ref{Lemma 2nd moment truncated intrinsic area}, Lemma \ref{Lemma absolute continuity for random metric space} applies with $\mathbf{P}=\check{\mathcal{P}}_1^{(K)}$ and $\nabla_{K}:=\zeta_{\sigma_K}-\zeta_{\sigma_K'}$, where $\sigma_K$ and $\sigma_K'$ are conditionally independent random leaves in $B_K$ with conditional law $\mathcal{A}_{K}(\cdot)/\mathcal{M}_K$ given $(\chi_u)_{u\in\mathbb{U}}$. We claim that when $K$ is large, $B_K=\partial\mathbb{U}$, that is
	\begin{align*}
		\lim_{K\to\infty}\mathcal{P}_1\left(\mathcal{A}=\mathcal{A}_K\right)=1.
	\end{align*}
	Indeed, as a direct consequence of Corollary 4 in \cite{B17}, we have that $\sup_{u\in\mathbb{U}}\chi_u^*<\infty$ $\mathcal{P}$-a.s., where $\chi_u^*:=\sup_{t\geq 0}\chi_u(t)$. By Lemma \ref{Lemma absolute continuity for random metric space}, it is hence enough to show that $\nabla_K$ has a bounded density in a neighbourhood of 0 to prove the Theorem.
	Similarly to Lemma \ref{Lemma density nabla integrable case}, we have
	
	\begin{lem}\label{Lemma density nabla truncated}
		Under $\check{\mathcal{P}}_1^{(K)}$, $\nabla_K$ has a density $h:\mathbb{R}\to\mathbb{R}_+\cup\left\{\infty\right\}$ given by 
		\begin{align*}
			h:x\mapsto&C\widehat{\mathbb{E}}_1\Bigg(\sum_{s>0}\mathds{1}_{\left\{\sup_{t\leq s}Y(t)\leq K\right\}}|\Delta_-Y(s)|^{\omega_-+\alpha}Y(s)^\alpha\\
		&\hspace{4cm}\times \int_0^\infty \mathrm{d}u g_{2,s}(u|\Delta_-Y(s)|^\alpha)g_{1,s}((u-x)Y(s)^\alpha)\Bigg),
		\end{align*}
		where $C$ is known and comes from the change of measure, and $g_{1,s},g_{2,s}$ are non-negative random functions, measurable with respect to the natural filtration of $Y$ at time $s$ for all $s\geq 0$, that are all pointwise bounded by $k$.
	\end{lem}
	
	The proof being very similar to that of Lemma 6, we do not provide all the steps, but only those where new arguments are needed.
	\begin{proof}
		We have the following analogue of \eqref{Equation second spine rescaled}:
	\begin{align*}
		\check{\mathcal{E}}_1^{(K)}\left(f(\nabla_K)\right)
		&=C\widehat{\mathcal{E}}_1\Bigg(\sum_{s>0}\mathds{1}_{\left\{\sigma^*\leq K\right\}}|\Delta_-\sigma(s)|^{\omega_-}\\
		&\hspace{1cm}\times\widehat{\mathcal{E}}_{1}\Big(\mathds{1}_{\left\{\widehat{\sigma}^*\leq K|\Delta_-\sigma(s)|^{-\omega_-}\right\}}f\left(\zeta_\sigma-s-|\Delta_-\sigma(s)|^{-\alpha}\zeta_{\widehat{\sigma}}\right)\Big|(\sigma(t))_{t\geq 0}\Big)\Bigg)\\
		&=C\widehat{\mathbb{E}}_1\Bigg(\sum_{s>0}\mathds{1}_{\left\{Y_1^*\leq K\right\}}|\Delta_-Y_1(s)|^{\omega_-}\\
		&\hspace{1cm}\times\widehat{\mathbb{E}}_{1}\left(\left.\mathds{1}_{\left\{Y_2^*\leq K|\Delta_-Y_1(s)|^{-\omega_-}\right\}}f\left(I_1-s-|\Delta_-Y_1(s)|^{-\alpha}I_2\right)\right|Y_1\right)\Bigg),\\
	\end{align*}
	by Lemma \ref{Lemma spine distributed as Y}, where $I_1$ and $I_2$ are the respective absorption times at 0 of two independent positive self-similar Markov processes $Y_1,Y_2$ with same distribution $\widehat{\mathbb{P}}_1$.
	The expectations in the right-hand side then becomes
	\begin{align*}
		\widehat{\mathbb{E}}_1\left(\sum_{s>0}\mathds{1}_{\left\{Y_1^*\leq K\right\}}|\Delta_-Y_1(s)|^{\omega_-}\int_0^\infty \mathrm{d}x g_{2,s}(x)f\left(I_1-s-|\Delta_-Y_1(s)|^{-\alpha}x\right)\right),
	\end{align*}
	where for any $s>0$ such that $\Delta_-Y_1(s)<0$, the random function
	\begin{align*}
		g_{2,s}:x\mapsto k(x)\widehat{\mathbb{P}}\left(\left.Y_2^*\leq K|\Delta_-Y_1(s)|^\alpha\right|I_2=x\right)
	\end{align*}
	is measurable with respect to the natural filtration of $Y_1$ at time $s$. Clearly, $g_{2,s}(x)\leq k(x)$ for all $x>0$. As before, applying \cite{DM2} Theorem 57, we obtain that
	\begin{align*}
		\check{\mathcal{E}}_1\left(f(\nabla_{|B})\right)
		&=C\widehat{\mathbb{E}}_1\Bigg(\sum_{s>0}\mathds{1}_{\left\{\sup_{t\leq s}Y_1(t)\leq K\right\}}|\Delta_-Y_1(s)|^{\omega_-}\int_0^\infty \mathrm{d}x g_{2,s}(x)\int_0^\infty\mathrm{d}yg_{1,s}(y)\\
		&\hspace{6cm}\times f\left(Y_1(s)^{-\alpha}y-|\Delta_-Y_1(s)|^{-\alpha}x\right)\Bigg),
	\end{align*}
	where $g_{1,s}$ is defined in the same way as $g_{2,s}$.	
	\end{proof}
	
	\begin{proof}[Proof of Theorem \ref{absolute continuity}, the case $\alpha>-\omega_-$ and $\omega_+/\omega_-\leq 2$.]
		As previously, we only need to show that $h$ given in Lemma \ref{Lemma density nabla truncated} is bounded to conclude by Lemma \ref{Lemma absolute continuity for random metric space}. Recall that the (random) functions $g_{1,s},g_{2,s}$, $s\geq 0$ in the definition of $h$ are bounded by $k$. For any $x\in\mathbb{R}$, similarly to \eqref{Equation upper bound h abs cont}, we thus have that
		\begin{align*}
			h(x)
			&\leq C\widehat{\mathbb{E}}_1\left(\sum_{s>0}\mathds{1}_{\left\{\sup_{t\leq s}Y(t)\leq K\right\}}|\Delta_-Y(s)|^{\omega_-}(Y(s)\vee|\Delta_-Y(s)|)^{\alpha}\right)||k||_\infty ||k||_1.
		\end{align*}
		As before, we can consider the simpler homogeneous case to show that $h$ is bounded, without loss of generality. Let $\eta(s) = \log(Y(s))$, $s>0$. We rewrite the expectation above as
	\begin{align*}
		&\widehat{\mathbb{E}}_1\Bigg(\sum_{s>0}\mathds{1}_{\left\{\sup_{t\leq s}\eta(t)\leq \log K\right\}}e^{(\omega_-+\alpha)\eta(s-)}(1-e^{\Delta_-\eta(s)})^{\omega_-}(e^{\Delta_-\eta(s)}\vee(1-e^{\Delta_-\eta(s)}))^\alpha\Bigg)\\
		&\hspace{1cm}\leq 2^{-\alpha}\widehat{\mathbb{E}}_1\Bigg(\sum_{s>0}\mathds{1}_{\left\{\sup_{t\leq s}\eta(t)\leq \log K\right\}}e^{(\omega_-+\alpha)\eta(s-)}(1-e^{\Delta_-\eta(s)})^{\omega_-}\Bigg)\\
		&\hspace{1cm}=2^{-\alpha}K^{\omega_-+\alpha-p}\widehat{\mathbb{E}}_1\Bigg(\sum_{s>0}e^{p\eta(s-)}(1-e^{\Delta_-\eta(s)})^{\omega_-}\Bigg),
	\end{align*}	 
	for any arbitrary fixed $p\in\intervalleoo{0}{\omega_+-\omega_-}$. The compensation formula then shows that the last expectation is equal to
	\begin{align*}
		\int_0^\infty\mathrm{d}se^{\kappa(\omega_-+q)s}\int_{\intervalleoo{-\infty}{0}}\Lambda(\mathrm{d}y)e^{\omega_-y}(1-e^{y})^{\omega_-},
	\end{align*}
	which is clearly finite from the choice of $q$ and the definition of $\omega_-$. We have thus proved that $h$ is bounded. By Lemma \ref{Lemma absolute continuity for random metric space}, this shows that $t\mapsto\mathcal{A}(\left\{\ell\in B:\zeta_\ell\leq t\right\})$ is absolutely continuous with respect to the Lebesgue measure $\check{\mathcal{P}}$-a.s., and therefore $\mathcal{P}$-a.s.,  and we conclude using that $\lim_{K\to\infty}\mathcal{P}(\mathcal{A}_{|B}=\mathcal{A})=1$ as stated earlier.
	\end{proof}

	\subsubsection{The singular case, $\alpha\leq -\omega_-$.}\label{subsubsection singular case}
	
	We finish the proof of Theorem \ref{absolute continuity}, that is we show that when $\alpha\leq\omega_-$, $\mathrm{d}A$ is almost surely singular with respect to the Lebesgue measure.
	
	Since $t\mapsto A(t)$ is non-decreasing, $A'$ exists almost surely, therefore by Fubini's Theorem we can define $A'$ for almost every $t$. For such $t$, applying \eqref{equation branching + self-similarity of A} and using the same notation we obtain:
	\begin{align*}
		A(t+\epsilon)-A(t)
		&\stackrel{d}{=}\sum_{i\geq 1}X_i^{\omega_-}(t)A_i(\epsilon X_i^\alpha(t)).
	\end{align*}
	Suppose that there are infinitely many fragments at time $t$. Then for all $n\geq 1$, set $\epsilon_n:=X_n^{-\alpha}(t)$ and divide the last expression by $\epsilon_n$ to get
	\begin{align*}
		\epsilon_n^{-1}(A(t+\epsilon_n)-A(t))
		&=\sum_{i\geq 1}X_n^\alpha(t)X_i^{\omega_-}(t)A_i(X_n^{-\alpha}(t)X_i^\alpha(t))\\
		&\geq X_n^{\alpha+\omega_-}(t)A_n(1).
	\end{align*}
	Since the $A_n$'s are i.i.d. copies of $A$, there are almost surely infinitely many $A_n(1)$'s which are greater than any given constant. But if $\alpha\leq -\omega_-$, we see that $\limsup_{n\to\infty}\epsilon_n^{-1}(A(t+\epsilon_n)-A(t))=\infty$ which is in contradiction with the fact that $A$ admits a derivative at $t$. This implies that there is only a finite number of fragments at time $t$, say $N\in\mathbb{N}$, therefore we can switch the sum and the limit and we obtain :
	\begin{align*}
		\lim_{\epsilon\to 0}\epsilon^{-1}(A(t+\epsilon)-A(t))
		&=\sum_{i\leq N}X_i^{\alpha+\omega_-}(t)A_i'(0),
	\end{align*}
	where the derivatives are well defined and equal to 0 by Lemma \ref{Lemma A(epsilon)}. Hence $A'(t)=0$ for almost every $t$ and $\mathrm{d}A$ is singular.

	\subsection{Number of fragments}\label{Subsection number of fragments}

	As we just saw in the proof of the singular case of Theorem \ref{absolute continuity}, there is a link between the number of fragments in the growth-fragmentation and the regularity of $A$. Even though we shall not use this relation later on, we state it in a corollary as it might be of independent interest. 
	
	\begin{cor}\label{Corollary number of fragments}
		Suppose that $\alpha>-\omega_-$, then almost surely the number of fragments with positive mass is infinite for every $t$ such that $a(t)>0$. 
		
		Conversely if $\alpha\leq-\omega_-$, then almost surely the number of fragments with positive mass is finite for almost every $t\geq 0$.
	\end{cor}
	
	\begin{proof}
	The second statement has been established in the subsection \ref{subsubsection singular case}.
	If $\alpha>-\omega_-$, then by Theorem \ref{absolute continuity} we have $\mathrm{d}A(t)=a(t)\mathrm{d}t$.
	Fix $s,t>0$ such that $t<\zeta$. By \eqref{equation branching + self-similarity of A} we can write
	\begin{align}
		A(t+s)-A(t)\stackrel{d}{=}\sum_{i=1}^{N_t}X_i(t)^{\omega_-} A_i(s X_i(t)^\alpha),
	\end{align}
	where $N_t$ is the number of fragments (possibly infinite) with positive mass at time $t$, $\left\{A_i;\ i=1..N_t\right\}$ are i.i.d. copies of $A$. Then using the above identity we have 
	\begin{align*}
		\frac{1}{\epsilon}\sum_{i=1}^{N_t}X_i(t)^{\omega_-} A_i(\epsilon X_i(t)^\alpha)
		&\xrightarrow[\epsilon\to 0]{\mathrm{a.s.}}a(t).
	\end{align*}
	Suppose moreover that $N_t<\infty$, then Lemma \ref{Lemma A(epsilon)} implies that $a(t)=0$.
	\end{proof}

	\section{Approximation of the profile}\label{Section Approximation of the density}
	
	Throughout this section, we assume that $\mathrm{d}A(t)=a(t)\mathrm{d}t$, or equivalently $\alpha>-\omega_-$, by Theorem \ref{absolute continuity}.
	Our main goal in what follows is to adapt the arguments of Haas \cite{H04} Section 5 to the growth-fragmentations case. We aim to show that $a$ can be approximated by both the small fragments and the relatively big ones. More precisely, define the processes $M,N$ for all $\epsilon>0,t\geq 0$ by
\begin{align*}
	M(t,\epsilon):=&\sum_{i\geq 1}X_i^{\omega_-}(t)\mathds{1}_{\left\{X_i(t)\leq\epsilon\right\}},\\
	N(t,\epsilon):=&\sum_{i\geq 1}\mathds{1}_{\left\{X_i(t)>\epsilon\right\}}.
\end{align*}

\begin{thm}\label{Theorem fragments approx density a}
	Suppose that $\alpha>-\omega_-$. Then for almost every $t\geq 0$, we have that
	\begin{align*}
		\epsilon^\alpha M(t,\epsilon)&\xrightarrow[\epsilon\to 0]{\mathrm{a.s.}}\frac{a(t)}{\alpha\kappa'(\omega_-)},\\
		\epsilon^{\omega_-+\alpha}N(t,\epsilon)&\xrightarrow[\epsilon\to 0]{\mathrm{a.s.}}\frac{a(t)}{(\omega_-+\alpha)|\kappa'(\omega_-)|}.
	\end{align*}
\end{thm}
Note that $\kappa'(\omega_-)\in\intervalleoo{-\infty}{0}$ by \eqref{Cramer hypothesis}.
Theorem \ref{Theorem fragments approx density a} is the analogue of Theorem 7 in \cite{H04}, which deals with self-similar fragmentations (see also \cite{D07} Proposition 4.2 addressing the profile of L\'evy trees associated with fragmentations that are not necessarily self-similar).

In order to prove this result, as Haas we shall focus on the small fragments, since the behaviour of $N(t,\epsilon)$ as $\epsilon\to 0$ can be deduced from that of $M(t,\epsilon)$ applying Tauberian's Theorems (as discussed in the end of this section).

\begin{lem}[Analogue of Lemma 8 in \cite{H04}]\label{Adaptation of Lemma 8 of Haas}
	Let $I$ be a random variable with density $k$, independent of $\mathbf{X}$. If $\alpha>-\omega_-$, then for almost every $t>0$,
	\begin{align*}
		\lim_{\epsilon\to 0}\epsilon^\alpha\mathcal{E}\left(M(t,\epsilon I^{\frac{1}{\alpha}})\Big|\mathbf{X}\right)\stackrel{a.s.}{=}a(t).
	\end{align*}	 
\end{lem}

Provided that $\mathbf{X}$ is a Feller process, the proof is almost identical to that of Haas, with the difference that one has to work on an event having a probability arbitrarily close to 1 and that ensures $\int_0^\infty a^2(t)\mathrm{d}t$ to have finite expectation (this event can be $H_K:=\{\sup_{u\in\mathbb{U}}\chi_u^*\leq K\}$ where $\chi_u^*:=\sup_{t\geq 0}\chi_u(t)$).
We skip the details of the proof of Lemma \ref{Adaptation of Lemma 8 of Haas}, the Feller property and its proof are given in Appendix.

The following lemma restates the first convergence in Theorem \ref{Theorem fragments approx density a}:

\begin{lem}\label{Lemma small fragments}
	When $\alpha>-\omega_-$, we have for almost every $t\in\mathbb{R}_+$ that
	\begin{align*}
		\epsilon^\alpha M(t,\epsilon)&\xrightarrow[\epsilon\to 0]{\mathrm{a.s.}}\ a(t)/\alpha\kappa'(\omega_-).
	\end{align*}
\end{lem}

\begin{proof}
	This proof is again very similar to that of Haas, we thus just focus on verifying that the hypotheses of the Wiener-Pitt theorem (Theorem 4.8.0 of \cite{BGT87}) are satisfied, and refer to the proof of \cite{H04} Theorem 7 to see how it applies to show the convergence of small fragments.
	What has to be shown is that the Mellin transform of $I$, defined as $\mathscr{M}_I(ix):=\widehat{\mathbb{E}}_1(I^{ix-1})$, exists and is non zero for all $x\in\mathbb{R}$.
	We already know by Lemma \ref{properties of k} that $\widehat{\mathbb{E}}_1(I^{-1})=\alpha\kappa'(\omega_-)\in\intervalleoo{0}{\infty}$. Let $\Psi$ be the characteristic exponent of $\eta$ defined as $\Psi:\theta\mapsto-\log\widehat{\mathbb{E}}(e^{i\theta\eta(1)})$.
Theorem 2.7(1) in \cite{PS16} shows that 
\begin{align*}
	\mathscr{M}_I(ix)
	=\frac{\Psi(-\alpha x)}{ix}\mathscr{M}_I(1+ix).
\end{align*}
It is not hard to check that $\Psi(-\alpha x)\neq 0$.
Moreover, it is also stated in the same theorem that
\begin{align*}
	\mathscr{M}_I(1+ix)=\Phi_+(0)\frac{\Gamma(1+ix)}{W_{\Phi_-}(1+ix)}W_{\Phi_+}(-ix),
\end{align*}
where $\Phi_+$ (respectively $\Phi_-$) is the characteristic exponent of the ascending (respectively descending) ladder height process of $\eta$ and $W_{\phi_+}$ (respectively $W_{\Phi_-}$) is the generalized Weierstrass product of $\Phi_+$ (respectively $\Phi_-$) as in $\cite{PS16}$ (see Kyprianou \cite{K14} or Bertoin \cite{B96} for definitions and details on ladder height processes). In particular, it is well-known that since $\eta$ drifts to $-\infty$, we have $\Phi_+(0)>0$. Therefore $W_{\phi_+}(z)\neq 0$ for all $z\in\mathbb{C}$ with $\mathrm{Re}(z)\geq 0$ by Theorem 3.2 of $\cite{PS16}$. Furthermore, it also ensures that $W_{\Phi_-}$ is holomorphic on $\left\{z\in\mathbb{C}:\mathrm{Re}(z)>0\right\}$ therefore $|W_{\Phi_-}(1+ix)|<\infty$. Since $\Gamma(1+ix)$ is non-zero, we see that $\mathscr{M}_I(1+ix)\neq 0$, we can then apply the Wiener-Pitt theorem, as planned, giving the claim.
\end{proof} 
	
To conclude the proof of Theorem \ref{Theorem fragments approx density a}, it remains to show that for almost every $t\geq 0$,
	\begin{align}\label{equation M sim N}
		\frac{-\alpha}{\omega_-}M(t,\epsilon)\underset{\epsilon\to 0}{\sim}\frac{\omega_-+\alpha}{\omega_-}\epsilon^{\omega_-}N(t,\epsilon).
	\end{align}
	Let $\mu:=\sum_{i\geq 1}\delta_{X_i(t)^{\omega_-}}$ and let $\overline{\mu}(x):=\mu(\intervalleoo{x}{\infty})$. Define $\mathrm{d}f(y)=y\mu(\mathrm{d}y)$. Equation \eqref{equation M sim N} can be shown using Tauberian's Theorems, we refer to the proof of equation (4) in \cite{B04} to see that $f$ is regularly varying at 0 with index $1-\beta\in\intervalleoo{0}{1}$ if and only if $\overline{\mu}$ is regularly varying at 0 with index $-\beta$. In that case, it holds that
	\begin{align*}
		\beta\epsilon\overline{\mu}(\epsilon)\underset{\epsilon\to 0}{\sim}(1-\beta)f(\epsilon).
	\end{align*}
	Now remark that $\overline{\mu}(\epsilon)=N(t,\epsilon^{1/\omega_-})$ and $f(\epsilon)=M(t,\epsilon^{1/\omega_-})$ which implies with Lemma \ref{Lemma small fragments} that $1-\beta=-\alpha/\omega_-\in\intervalleoo{0}{1}$. This proves that \eqref{equation M sim N} holds.

	\section{Hausdorff dimension}\label{Section Hausdorff dimension}

	We now study the case $\alpha\leq-\omega_-$ so that $\mathrm{d}A$ is singular with respect to the Lebesgue measure almost surely, by Theorem \ref{absolute continuity}. We describe the set on which $\mathrm{d}A$ is concentrated through its Hausdorff dimension, see \cite{F85} for background.
	Recall by Lemma \ref{Lemma A(epsilon)} that $A$ is strictly increasing on $\intervalleoo{0}{\zeta}$ so the support of $\mathrm{d}A$ is exactly $\intervalleff{0}{\zeta}$. However $\mathrm{dim}_H(\mathrm{d}A)$ is not necessarily 1, as shown in the following theorem:

	\begin{thm}\label{Theorem Hausdorff dimension}
		Suppose $-\omega_+<\alpha\leq -\omega_-$, then it holds that:
		\begin{align*}
			\mathrm{dim}_H(\mathrm{d}A)=\frac{\omega_-}{-\alpha},\qquad \mathcal{P}_1\text{-a.s}.
		\end{align*}
		Furthermore, $\mathrm{dim}_H(\mathrm{d}A)\geq -\omega_-/\alpha$ holds for any value of $\alpha\leq\omega_-$.
	\end{thm}
	
	\begin{rem}[Hölder continuity]
		Copying the argument of Haas \cite{H04} Proposition 12(i), Theorem \ref{Theorem Hausdorff dimension} directly implies that if $-\omega_+<\alpha\leq -\omega_-$, then $A$ is $\gamma$-Hölder continuous for every $\gamma<-\omega_-/2\alpha$.
	\end{rem}
	
	\subsection{The lower bound}
	
	Frostman's Lemma (see e.g. \cite{F85} Corollary 6.6(a)), that we now recall, is the key to the lower bound.

	\begin{lem}[Frostman's Lemma]\label{Lemma Frostman Lemma}
	Let $b\in\intervalleof{0}{1}$ and let $\mu$ be a finite measure on $\mathbb{R}$. If
	\begin{align*}
		\mathcal{I}_b(\mu):=\int_{\mathbb{R}}\int_{\mathbb{R}}\frac{\mathrm{d}\mu(u)\mathrm{d}\mu(v)}{|u-v|^b}<\infty,
	\end{align*}
	then $\dim_H(\mu)\geq b$.		 
	\end{lem}
	
	\begin{proof}[\textbf{Proof of Theorem \ref{Theorem Hausdorff dimension}: the lower bound}]
	
	In the light of Lemma \ref{Lemma Frostman Lemma}, it is sufficient to show that
	\begin{align}\label{equation b-energie}
		\mathcal{E}_1(\mathcal{I}_b(\mathrm{d}A))<\infty,
	\end{align}
	for all $b<\frac{\omega_-}{-\alpha}$.
	We write
\begin{align*}
	\mathcal{E}_1(\mathcal{I}_b(\mathrm{d}A))
	&=\mathcal{E}_1\left(\int_0^\infty\int_0^\infty\frac{\mathrm{d}A(u)\mathrm{d}A(v)}{|u-v|^b}\right).\\
	\shortintertext{As previously, we sample a first spine applying Lemma \ref{Lemma A and random leaf} and we get}
	\mathcal{E}_1(\mathcal{I}_b(\mathrm{d}A))
	&=\widehat{\mathcal{E}}_1\left(\mathcal{E}_1\left(\int_0^\infty\frac{\mathrm{d}A(v)}{|\zeta_{\sigma}-v|^b}\Big|(\sigma(t))_{t\leq \zeta_{\sigma}}\right)\right)\\
	&=\widehat{\mathbb{E}}_1\left(\mathcal{E}_1\left(\int_0^\infty\frac{\mathrm{d}A(v)}{|I_1-v|^b}\Big|(Y(t))_{t\leq I_1}\right)\right),
\end{align*}
	where $I_1$ denotes the life time of $Y$. We then decompose $A$ as in the proof of Theorem \ref{absolute continuity}, we write
	\begin{align*}
	\mathcal{E}_1(\mathcal{I}_b(\mathrm{d}A))
	&=\widehat{\mathbb{E}}_1\left(\sum_{s<I_1}\mathcal{E}_{|\Delta_- Y(s)|}\left(\int_s^\infty\frac{\mathrm{d}A_s(v)}{|I_1-v|^b}\Big|(Y(t))_{t\leq I_1}\right)\right),
\end{align*}
	where $A_s$ is the intrinsic area function associated with the restriction of $\mathcal{A}$ to $\partial\mathbb{U}_s$, the leaves of the subtree generated by the cell born at time $s$. In particular, $A_s$ has same conditional distribution as $A$ under $\mathcal{P}_{|\Delta_- Y_s|}$ shifted by $s$ (see \cite{BBCK16} Theorem 4.7). We rewrite the right hand-side above denoting $A_s^*:=A_s(s+\cdot)$ and get
\begin{align*}
	&\mathcal{E}_1(\mathcal{I}_b(\mathrm{d}A))\\
	&\hspace{1cm}=\widehat{\mathbb{E}}_1\left(\sum_{s<I_1}\mathcal{E}_{|\Delta_- Y(s)|}\left(\int_0^\infty\frac{\mathrm{d}A_s^*(v)}{|I_1-s-v)|^b}\Big|(Y(t))_{t\leq I_1}\right)\right)\\
	&\hspace{1cm}=\widehat{\mathbb{E}}_1\left(\sum_{s<I_1}|\Delta_- Y(s)|^{\omega_-}\mathcal{E}_1\left(\int_0^\infty\frac{\mathrm{d}A_s^*(v)}{|I_1-s-|\Delta_- Y(s)|^{-\alpha}v|^b}\Big|(Y(t))_{t\leq I_1}\right)\right),
\end{align*}
	where we applied the self-similarity of $A_s^*$ for the last equality. Using Lemmas \ref{Lemma A and random leaf} and \ref{Lemma spine distributed as Y}, let $I_2$ be a random variable with density $k$ independent of $Y$ and write
\begin{align*}
	\mathcal{E}_1(\mathcal{I}_b(\mathrm{d}A))	
	&=\widehat{\mathbb{E}}_1\left(\sum_{s<I_1}|\Delta_- Y(s)|^{\omega_-}\widehat{\mathbb{E}}_1\left(\big|I_1-s-|\Delta_- Y(s)|^{-\alpha}I_2\big|^{-b}\Big|(Y(t))_{t\leq I_1}\right)\right)\\
	&=\widehat{\mathbb{E}}_1\left(\sum_{s<I_1}|\Delta_- Y(s)|^{\omega_-}\big|I_1-s-|\Delta_- Y(s)|^{-\alpha}I_2\big|^{-b}\right).
\end{align*}
	We use Theorem 57 of \cite{DM2}, justified again by positivity: the optional projection of $s\mapsto|I_1-s-|\Delta_-Y_s|^{-\alpha}I_2|^{-b}$ being given by
	\begin{align*}
		s\mapsto \int_0^\infty\mathrm{d}u\int_0^\infty\mathrm{d}v\frac{k(u)k(v)}{|Y(s)^{-\alpha}u-|\Delta_-Y(s)|^{-\alpha}v|^b}.
	\end{align*}
	We hence obtain
	\begin{align*}
		\mathcal{E}_1(\mathcal{I}_b(A))
		&=\widehat{\mathbb{E}}_1\left(\sum_{s<I_1}|\Delta_-Y(s)|^{\omega_-}\int_0^\infty\mathrm{d}u\int_0^\infty\mathrm{d}v\frac{k(u)k(v)}{|Y(s)^{-\alpha}u-|\Delta_-Y(s)|^{-\alpha}v|^b}\right).
	\end{align*}
	Assuming without loss of generality that $Y$ is homogeneous, we see that the latter is equal to
	\begin{align*}
		&\widehat{\mathbb{E}}_1\Bigg(\sum_{s>0}e^{(\omega_-+\alpha b)\eta(s-)}(1-e^{\Delta_-\eta(s)})^{\omega_-}\\
		&\qquad\qquad\qquad\qquad\times\int_0^\infty\mathrm{d}u\int_0^\infty\mathrm{d}v\frac{k(u)k(v)}{|e^{-\alpha\Delta_-\eta(s)}u-(1-e^{\Delta_-\eta(s)})^{-\alpha}v|^b}\Bigg).
	\end{align*}
	Lemma \ref{Lemma technical inequalities} in Appendix finally yields that
	\begin{align*}	
		\mathcal{E}_1(\mathcal{I}_b(A))
		&\leq C\widehat{\mathbb{E}}_1\left(\sum_{s>0}e^{(\omega_-+\alpha b)\eta(s-)}(1-e^{\Delta_-\eta(s)})^{\omega_-}\right),
	\end{align*}
	where $C$ is a deterministic finite constant. The compensation formula then shows that the righ-hand side above is equal to
	\begin{align*}
		C\int_0^\infty e^{s\kappa(2\omega_-+\alpha b)}\mathrm{d}s\int_{\mathbb{R}_-}(1-e^y)^{\omega_-}e^{\omega_-y}(\Lambda+\widetilde{\Lambda})(\mathrm{d}y).
	\end{align*}
	The last integral being finite by definition of $\omega_-$, we see that this expression is finite whenever $\kappa(2\omega_-+\alpha b)<0$, which is the case if and only if $b\in\intervalleoo{\frac{2\omega_--\omega_+}{-\alpha}}{\frac{\omega_-}{-\alpha}}$ (this interval is never empty since we assume that $\omega_-<\omega_+$). We thus have shown that
	\begin{align*}
		b\in\intervalleoo{\frac{2\omega_--\omega_+}{-\alpha}}{\frac{\omega_-}{-\alpha}}\quad\Rightarrow\quad\eqref{equation b-energie}\quad\Rightarrow\quad\mathcal{I}_b(A)<\infty\quad\text{a.s.},
	\end{align*}
	which by Lemma \ref{Lemma Frostman Lemma} gives the lower bound for any $\alpha\leq\omega_-$.

	\end{proof}

	\subsection{The upper bound}
	
	In the pure fragmentation setting, the analogue of $A$ is the function $M$ of the loss of mass. The upper bound of $\mathrm{dim}_H(\mathrm{d}M)$ has been obtained by Haas and Miermont in \cite{HM04} by constructing the CRT induced by the fragmentation. They first investigated the Hausdorff dimension of the leaves of the tree, then they deduced the upper bound for $\mathrm{dim}_H(\mathrm{d}M)$ using the fact that the image of a set by any surjective Lipschitz mapping (in their case the cumulative height profile) has Hausdorff dimension at most equal to that of the original set (this is a direct consequence of Lemma 1.8 in \cite{F85}). Since Rembart and Winkel \cite{RW16} already provided the Hausdorff dimension of the leaves $\mathcal{L}(\mathcal{T})$ of the CRT, we can use the same argument as Haas and Miermont to obtain the upper bound. For this reason we now work on $(\mathcal{T},\mathcal{A}_{\mathcal{T}})$ instead of $(\overline{\mathbb{U}},\mathcal{A})$.
	
	It is not hard to see from its definition in Section \ref{Section preliminaries} that the height function $\mathrm{ht}$ is Lipschitz with respect to the metric $d$.

	\begin{proof}[\textbf{Proof of the upper bound}]
		Recall that $\mathcal{A}_{\mathcal{T}}$ is supported on $\mathcal{L}(\mathcal{T})$ (more precisely the subset of $\mathcal{L}(\mathcal{T})$ corresponding to leaves in $\partial\mathbb{U}$). 
		By definition of $A$ \eqref{equation def A}, $\mathrm{d}A(\mathrm{ht}(\mathcal{L}(\mathcal{T})))$ is equal to its total mass $\mathcal{M}$. Therefore,
		\begin{align*}
			\mathrm{dim}_H(\mathrm{d}A)\leq\mathrm{dim}_H(\mathrm{ht}(\mathcal{L}(\mathcal{T})))\leq\mathrm{dim}_H(\mathcal{L}(\mathcal{T}))
		\end{align*}
		since $\mathrm{ht}$ is Lipschitz. By Theorem 4.5 in \cite{RW16}, $\mathrm{dim}_H(\mathcal{L}(\mathcal{T}))=-\omega_-/\alpha$, which gives the claim.
	\end{proof}

\section{Application to Boltzmann random planar maps}
In \cite{BBCK16}, the authors showed that by cutting particular Boltzmann random maps at heights, one obtains a collection of cycles whose lengths are described in scaling limit by a specific family of growth-fragmentations with cumulant function of the form
\begin{align*}
	\kappa_\theta(q):=\frac{\cos(\pi(q-\theta))}{\sin(\pi(q-2\theta))}\cdot\frac{\Gamma(q-\theta)}{\Gamma(q-2\theta)},\qquad q\in\intervalleoo{\theta}{2\theta+1},
\end{align*}
with self-similarity index $\alpha=1-\theta$, for some parameter $\theta\in\intervalleof{1}{3/2}$. (The case $\theta=3/2$ corresponds to the Brownian map.)

The Cram\'er hypothesis \eqref{Cramer hypothesis} holds with $\omega_-=\theta+1/2$ and $\omega_+=\theta+3/2$, so that the intrinsic area of the ball of radius $r$ is an absolutely continuous function of $r$, by Theorem \ref{absolute continuity}. The small cycle lengths in the random maps are related to $a$ by Theorem \ref{Theorem fragments approx density a} in this paper and Theorem 6.8 in \cite{BBCK16}.

\section*{Acknowledgement}

I would like to thank Jean Bertoin for his constant support, thorough supervision and genuine kindness. I am also grateful to Bastien Mallein for numerous helpful discussions.

\section*{Appendix}

We state and prove here a technical results on the density $k$ that we have used in the proof of the lower bound in Theorem \ref{Theorem Hausdorff dimension}. 

	\begin{lem}\label{Lemma technical inequalities}
		Let $b,c\in\intervalleoo{0}{1}$. We have that
			\begin{align*}
				\int_0^{\infty}\mathrm{d}u\int_0^{\infty}\mathrm{d}v\frac{k(u)k(v)}{|uc^{-\alpha}-v(1-c)^{-\alpha}|^b}
				&\leq C,
			\end{align*}
			where $C$ is a finite constant not depending on $c$.
	\end{lem}
	
	\begin{proof}
	We have
	\begin{align*}
		\int_0^{\infty}\mathrm{d}u\int_0^{\infty}\mathrm{d}v&\frac{k(u)k(v)}{|uc^{-\alpha}-v(1-c)^{-\alpha}|^b}\\
		&\qquad\qquad=c^\alpha(1-c)^\alpha\int_0^\infty k(uc^\alpha)\mathrm{d}u\int_0^\infty \mathrm{d}v\frac{k(v(1-c)^\alpha)}{|u-v|^b}.
	\end{align*}
	Consider the last integral, we have that
	\begin{align*}
		\int_{0}^{\infty}\frac{k(v(1-c)^\alpha)}{|u-v|^b}\mathrm{d}v
		&\leq\int_{u-1}^{u+1}\frac{||k||_\infty}{|u-v|^b}\mathrm{d}v+\int_0^\infty k(v(1-c)^\alpha)\mathrm{d}v\\
		&\leq\int_{u-1}^{u+1}\frac{||k||_\infty}{|u-v|^b}\mathrm{d}v+\int_0^\infty k(v)\mathrm{d}v\\
		&=C,
	\end{align*}
	where $C$ denotes a constant not depending on $c$.
	This yields that
	\begin{align*}
		\int_0^{\infty}\mathrm{d}u\int_0^{\infty}\mathrm{d}v\frac{k(u)k(v)}{|uc^{-\alpha}-v(1-c)^{-\alpha}|^b}
		&\leq c^\alpha(1-c)^\alpha\int_0^\infty Ck(uc^\alpha)\mathrm{d}u\\
		&=(1-c)^\alpha C
	\end{align*}
	Notice that the same arguments apply when the roles of $c$ and $(1-c)$ are exchanged, which entails that the upper bound that we just obtained holds with $(c\vee(1-c))^\alpha$ instead of $(1-c)^\alpha$. One remarks that $c\vee(1-c)\geq 1/2$ and the claim follows.
	\end{proof}

	For $q>0$, we define $\ell^{q\downarrow}$ the subset of $\ell^q$ of non-increasing null sequences with finite $q$-norm, denoted by $||\cdot||_q$.
	
\begin{lem}[Feller's Property]\label{Lemma Feller}
		The law of the growth-fragmentation $\mathbf{X}$ satisfies the following Feller's property: let $\underline{x}_n$, $n\in\mathbb{N}$ and $\underline{x}$ be elements of $\ell^{\omega_-\downarrow}$ such that $(\underline{x}_n)_{n\geq 1}$ converges in $\ell^{\omega_-\downarrow}$ to $\underline{x}$. Then it holds that 
		\begin{align*}
			\mathcal{P}_{\underline{x}_n}\Rightarrow\mathcal{P}_{\underline{x}},\quad\text{as }n\to\infty,
		\end{align*}
		where $\Rightarrow$ means weak convergence in the sense of finite-dimensional distributions in $\ell^{\omega_-\downarrow}$.
\end{lem}
	
	\begin{proof}[Proof of Lemma \ref{Lemma Feller}]
		We denote $\underline{x}_n:=(x_{n,1},x_{n,2},\cdots)$ and $\underline{x}:=(x_1,x_2,\cdots)$.
		Let $\mathbf{X}^{(n)}$ (respectively $\mathbf{Y}$) be a self-similar growth-fragmentation process with distribution $\mathcal{P}_{\underline{x}_n}$ (respectively $\mathcal{P}_{\underline{x}}$). We shall show that the Wasserstein distance between $\mathcal{P}_{\underline{x}_n}$ and $\mathcal{P}_{\underline{x}}$ converges to zero, which will entail the claim.
		Let $t>0$ and write
		\begin{align}\label{equation decomposition GF's}
			\mathcal{E}\left(||\mathbf{X}^{(n)}(t)-\mathbf{Y}(t)||_{\omega_-}^{\omega_-}\right)
			&\leq\mathcal{E}\left(\sum_{k\geq 1}||\mathbf{X}^{(n)}_k(t)-\mathbf{Y}_k(t)||_{\omega_-}^{\omega_-}\right),
		\end{align}
		where $\mathbf{X}^{(n)}_k$ and $\mathbf{Y}_k$ are growth-fragmentations with respective distributions $\mathcal{P}_{x_{n,k}}$, $\mathcal{P}_{x_k}$. In the same vein as in the proof of Proposition 2 in \cite{B17} (viz the branching property for growth-fragmentations), we fix $\epsilon>0$ and define $\mathbf{X}^{(n)}_{k,\epsilon}$, (respectively $\mathbf{Y}_{k,\epsilon}$) the growth-fragmentation obtained from $\mathbf{X}^{(n)}_k$ (respectively $\mathbf{Y}_k$) by killing every fragment - and those it generates in the future - as soon as it reaches a size smaller than $\epsilon$. The triangle inequality on the right-hand side of \eqref{equation decomposition GF's} entails that
		\begin{align}\label{equation middle term GF's}
			\mathcal{E}\left(||\mathbf{X}^{(n)}(t)-\mathbf{Y}(t)||_{\omega_-}^{\omega_-}\right)
			\leq 3^{\omega_-}\left(A_n+B_n+C\right),
		\end{align}
		where
		\begin{align*}
			A_n:=&\mathcal{E}\left(\sum_{k\geq 1}||\mathbf{X}^{(n)}_k(t)-\mathbf{X}^{(n)}_{k,\epsilon}(t)||_{\omega_-}^{\omega_-}\right)\\
			B_n:=&\mathcal{E}\left(\sum_{k\geq 1}||\mathbf{X}^{(n)}_{k,\epsilon}(t)-\mathbf{Y}_{k,\epsilon}(t)||_{\omega_-}^{\omega_-}\right)\\
			C:=&\mathcal{E}\left(\sum_{k\geq 1}||\mathbf{Y}_{k,\epsilon}(t)-\mathbf{Y}_k(t)||_{\omega_-}^{\omega_-}\right).
		\end{align*}
		Recall that $X_j(t)$ is the size of the $j$th largest fragment in a growth-fragmentation at time $t$. We define $X_j^*(t)$ the infimum of the sizes of the ancestors of $X_j(t)$ before time $t$. In particular, if $X_j(t)>0$, then $X_j^*(t)>0$.
		Fix $a>\epsilon$ and write
		\begin{align*}
			A_n
			=\sum_{k\geq 1}\mathds{1}_{\left\{x_{n,k}\leq\epsilon\right\}}&\mathcal{E}_{x_{n,k}}\left(||\mathbf{X}_k^{(n)}(t)||_{\omega_-}^{\omega_-}\right)\\
			+\mathds{1}_{\left\{x_{n,k}>\epsilon\right\}}&\mathcal{E}_{x_{n,k}}\left(\sum_{j\geq 1}X_j(t)^{\omega_-}\mathds{1}_{\left\{0<X_j^*(t)\leq\epsilon,\ X_j(t)\leq a\right\}}\right)\\
			+\mathds{1}_{\left\{x_{n,k}>\epsilon\right\}}&\mathcal{E}_{x_{n,k}}\left(\sum_{j\geq 1}X_j(t)^{\omega_-}\mathds{1}_{\left\{0<X_j^*(t)\leq\epsilon,\ X_j(t)>a\right\}}\right).
		\end{align*}
		The first part is smaller than $\sum_{k\geq 1}x_{n,k}^{\omega_-}\mathds{1}_{\left\{x_{n,k}\leq\epsilon\right\}}$ by Theorem 2 in \cite{B17}. Applying Proposition 4.6 and Theorem 4.7 in \cite{BBCK16}, we bound the second part by 
		\begin{align*}
			&\sum_{k\geq 1}\mathds{1}_{\left\{x_{n,k}>\epsilon\right\}}\mathcal{E}_{x_{n,k}}\left(\sum_{j\geq 1}X_j(t)^{\omega_-}\mathds{1}_{\left\{X_j(t)\leq a\right\}}\right)\\
			&\hspace{3cm}\leq\sum_{k\geq 1}\mathds{1}_{\left\{x_{n,k}>\epsilon\right\}}x_{n,k}^{\omega_-}\widehat{\mathbb{P}}_{x_{n,k}}\left(0<Y(t)\leq a\right).
		\end{align*}
		Hence, since $Y$ is a Feller process, we have that
		\begin{align}\label{equation limsup A_n}
			\limsup_{n\to\infty}A_n
			&\leq\sum_{k\geq 1}\mathds{1}_{\left\{x_k\leq\epsilon\right\}}x_k^{\omega_-}+\mathds{1}_{\left\{x_k\geq\epsilon\right\}}x_k^{\omega_-}\widehat{\mathbb{P}}_{x_k}\left(0<Y(t)\leq a\right)\nonumber\\
			&\hspace{2.5cm}+\mathds{1}_{\left\{x_k\geq\epsilon\right\}}\mathcal{E}_{x_k}\left(\sum_{j\geq 1}X_j(t)^{\omega_-}\mathds{1}_{\left\{0<X_j^*(t)\leq\epsilon,\ X_j(t)\geq a\right\}}\right)
		\end{align}
		(the latter expectations follow from self-similarity, Fatou's Lemma and stochastic continuity).
		The sum of the terms smaller than $\epsilon$ can be taken arbitrarily close to 0, which is also the case of the second part of the sum with $a$, by dominated convergence.
		To see that the last terms of\eqref{equation limsup A_n} can also be taken as small as one wishes with $\epsilon$ and $a$, we refer to the proof of Proposition 2 in \cite{B17}. Therefore, it holds that $\lim_{n\to\infty}A_n=0$. Similar arguments can be used to deal with $C$.
		
		It remains to show that $\lim_{n\to\infty}B_n=0$ and the claim will follow from \eqref{equation middle term GF's}. Using the self-similarity, we have that
		\begin{align*}
			B_n
			&=\mathcal{E}\left(\sum_{k\geq 1}||\mathds{1}_{\left\{x_{n,k}>\epsilon\right\}}x_{n,k}\mathbf{X}^{(n)}_{k,\epsilon/x_{n,k}}(tx_{n,k}^\alpha)-\mathds{1}_{\left\{x_k>\epsilon\right\}}x_k\mathbf{Y}_{k,\epsilon/x_k}(tx_k^\alpha)||_{\omega_-}^{\omega_-}\right).
		\end{align*}
		Now each growth-fragmentation has distribution $\mathcal{P}_1$, and since the Wasserstein metric is given by the infimum over the set of joint distributions, we can assume that $\mathbf{Y}_k=\mathbf{X}^{(n)}_k$. We drop the indicator functions, letting $\mathbf{Y}_{k,\delta}$ be the null sequence whenever $\delta\geq 1$. Since $\underline{x}_n\xrightarrow{\ell^{\omega_-}}\underline{x}$, we just need to show the following convergence:
		\begin{align}\label{equation B_n to 0}
			\sum_{k\geq 1}x_k^{\omega_-}\mathcal{E}\left(||\mathbf{Y}_{k,\epsilon/x_{n,k}}(tx_{n,k}^\alpha)-\mathbf{Y}_{k,\epsilon/x_k}(tx_k^\alpha)||_{\omega_-}^{\omega_-}\right)\underset{n\to\infty}{\longrightarrow}0.
		\end{align}
		The left-hand side is bounded by
		\begin{align*}
			&\sum_{k\geq 1}x_k^{\omega_-}\mathcal{E}\Big(||\mathbf{Y}_{k,\epsilon/x_{n,k}}(tx_{n,k}^{\alpha})-\mathbf{Y}_{k,\epsilon/x_k}(tx_{n,k}^{\alpha})||_{\omega_-}^{\omega_-}\\
			&\hspace{6cm}+||\mathbf{Y}_{k,\epsilon/x_k}(tx_{n,k}^{\alpha})-\mathbf{Y}_{k,\epsilon/x_k}(tx_k^{\alpha})||_{\omega_-}^{\omega_-}\Big).
		\end{align*}
		The second part converges to 0 as $n\to\infty$ by stochastic continuity. The first part contains only fragments whose ancestors have minimum size between $\epsilon/(x_{n,k}\vee x_k)$ and $\epsilon/(x_{n,k}\wedge x_k)$. The dominated convergence theorem yields the claim, that is \eqref{equation B_n to 0} holds, which concludes the proof. 
	\end{proof}

	\bibliographystyle{plain}
	\bibliography{biblio}		
		
\end{document}